\documentclass[a4paper,12pt,reqno]{amsart}

\usepackage{a4wide}
\usepackage{amsmath}
\usepackage[english]{babel}
\usepackage{amsfonts}
\usepackage{amssymb}
\usepackage{dsfont}
\usepackage{bm}
\usepackage{mathrsfs}
\usepackage{graphicx}
\usepackage{palatino}
\usepackage{fancyhdr}
\usepackage{amsthm}
\usepackage{empheq}
\usepackage{cases}
\usepackage[all]{xy}
\usepackage{stmaryrd}
\usepackage[colorlinks, citecolor=blue, linkcolor=red]{hyperref}

\usepackage[]{draftwatermark}
\SetWatermarkText{${\mathcal{DRAFT}}$}
\SetWatermarkScale{6}

\usepackage{mathtools}
\mathtoolsset{showonlyrefs,showmanualtags}

\def\NN{{{\mathbb N}}}

\def\RR{{\mathbb R}}

\newcommand{\eps}{\varepsilon}

\newtheorem{theorem}{{Theorem}}[section]
\newtheorem{lemma}[theorem]{{Lemma}}
\newtheorem{proposition}[theorem]{{Proposition}}
\newtheorem{corollary}[theorem]{{Corollary}}

\theoremstyle{definition}

\newtheorem{remark}[theorem]{{Remark}}

\numberwithin{equation}{section}

\begin{document}

\title[A Triviality Result for Semilinear Parabolic Equations]{A Triviality Result for Semilinear Parabolic Equations}

\author{Daniele Castorina}
\address{Daniele Castorina\\
	Dipartimento di Matematica e Applicazioni,
	Universit\`a di Napoli, Via Cintia, Monte S. Angelo 80126 Napoli,
	Italy}
\email{daniele.castorina@unina.it}

\author{Giovanni Catino}
\address{Giovanni Catino\\
	Dipartimento di Matematica,
	Politecnico di Milano, Piazza Leonardo da Vinci 32, 20133, Milano,
	Italy}
\email{giovanni.catino@polimi.it}

\author{Carlo Mantegazza}
\address{Carlo Mantegazza\\
	Dipartimento di Matematica e Applicazioni,
	Universit\`a di Napoli, Via Cintia, Monte S. Angelo 80126 Napoli,
	Italy}
\email{c.mantegazza@sns.it}
\date{\today}

\begin{abstract}  We show a triviality result for  ``pointwise'' monotone in time, bounded ``eternal'' solutions of the semilinear heat equation 
\begin{equation*}
u_{t}=\Delta u + |u|^{p} 
\end{equation*}
on complete Riemannian manifolds of dimension $n \geq 5$ with nonnegative Ricci tensor, when $p$ is smaller than the critical Sobolev exponent $\frac{n+2}{n-2}$.
\end{abstract}

%\keywords{???, ???}

\maketitle

\section{Introduction}

Let $(M,g)$ be a complete Riemannian manifold with nonnegative Ricci tensor. We are going to consider bounded ``eternal'' classical solutions $u\in C^{2,1}(M\times\RR)$ of the semilinear heat equation 
\begin{equation}\label{eq-pde}
u_{t}=\Delta u + |u|^{p} 
\end{equation}
for $p>1$ and $T\leq +\infty$, with particular attention to ``triviality'', that is, to conditions forcing the solutions to be identically zero.

A reason for the interest in such eternal (or also ``ancient'') solutions is that they typically arise as blow--up limits when the solutions of semilinear parabolic equations in time intervals as $[0,T)$ develop a singularity at a certain time $T \in \mathbb{R}$, i.e. the solution $u$ becomes unbounded as $t$ goes to $T$. 

In the Euclidean space, it is well known (see~\cite{tal} and~\cite[Proposition~B]{guiniwang}) that noncostant, positive global radial (static, solution of the elliptic problem) solutions on $\mathbb{R}^n \times \mathbb{R}$ exist for the ``critical'' or any ``supercritical'' exponent $p \geq p_S=\frac{n+2}{n-2}$ (this latter is the {\em critical Sobolev exponent}), hence providing a counterexample to triviality in such case. In the non--static (parabolic) situation, while triviality of eternal {\em radial} (parabolic) positive solutions can be shown in the range of subcritical exponents $1<p<p_S$, the same expected result for general (not necessarily radial) solutions is known only in the range $1<p< \frac{n(n+2)}{(n-1)^2}$ (see Remark~\ref{rem-bv} below). Indeed, such triviality when $\frac{n(n+2)}{(n-1)^2} \leq p< p_S$ is a quite long standing open problem  (see~\cite{pqs1,pqs2}), which might have been finally solved by Quittner in the preprint~\cite{quittner2}, appeared during the redaction of this work.

We mention that these triviality issues for eternal (and also ancient) solutions have been recently addressed and partially extended to the cases of compact or bounded geometry Riemannian manifolds, by the first and third author in~\cite{cama2, cama1}.

Our aim in this paper is to show the following triviality theorem for eternal solutions, ``pointwise'' monotone in time (mentioned without proof in~\cite[Remark~4.3\,(b)]{pqs2} for $M=\RR^n$).

\begin{theorem}\label{t-triv}
Let $(M,g)$ be a complete Riemannian manifold of dimension $n \geq 5$ with nonnegative Ricci tensor. Let $u\in C^{2,1}(M\times\RR)$ be a {\em bounded} eternal solution of equation~\eqref{eq-pde} with $u_t \geq 0$ and $1<p<p_S$. Then, $u\equiv 0$.\end{theorem}

\begin{remark}\label{rem-bv} We observe that, essentially with the same proof of~\cite{bi-ve}, we can prove the following result, extending to manifolds the analogous one in $\RR^n$.

\bigskip

{\em 
\noindent Let $(M,g)$ be a complete Riemannian manifold of dimension $n \geq 3$ with nonnegative Ricci tensor. Let $u\in C^{2,1}(M\times\RR)$ be an eternal solution of equation~\eqref{eq-pde} with 
$$
1<p<\frac{n(n+2)}{(n-1)^2}.
$$
Then, $u\equiv 0$.
}

\bigskip

Indeed, as observed in Remark~\ref{r-bv}, Lemma~3.1 in~\cite{bi-ve-ra} and thus also Lemma~3 and Lemma~4 in~\cite{bi-ve}, hold true on any Riemannian manifolds with nonnegative Ricci tensor (since are based only on Bochner inequality). In particular, using good cutoff functions that can be constructed on manifolds with nonnegative Ricci tensor (see Remark~\ref{r-cutoff}), from Lemma~4 in~\cite{bi-ve}, given $(x,t)\in M\times \RR$, for every $\rho>0$ and every 
$$
1<p<\frac{n(n+2)}{(n-1)^2},
$$ 
there exists a constant $C=C(n,p)$ such that  
$$
\int_{B(x,\rho)\times(t-\rho^{2},t+\rho^{2})} u^{2p} \leq C\rho^{2-\frac{4p}{p-1}}\mu(B({x},\rho)),
$$
where $B(x,\rho)\subseteq M$ is the metric ball with center $x$ and radius $\rho$ and $\mu$ the canonical Riemannian measure of $(M,g)$. By Bishop--Gromov  inequality (see~\cite{gahula}), we then obtain
$$
\int_{B(x,\rho)\times(t-\rho^{2},t+\rho^{2})} u^{2p} \leq C\rho^{2-\frac{4p}{p-1}}\mu(B({x},\rho))\leq C \rho^{n+2-\frac{4p}{p-1}}.
$$
Since $p<p_S$, letting $\rho\to+\infty$ we conclude that $u=0$ in the whole $M\times \RR$.
\end{remark}

\section{Integral Estimates}\label{grdest}

The following Hamilton--type gradient estimate can be shown analogously to Lemma~3.1 in~\cite{cama2} and will be crucial in the proof of the integral estimate in Proposition~\ref{p-lemma3}.

\begin{lemma}\label{p-grad}
Let $u\in C^{2,1}(\Omega\times(0,T))$ be a positive bounded solution of equation~\eqref{eq-pde} with $0<u\leq D$ and let $\rho>0$ such that $\overline{B}(\overline{x},2\rho)\times[\overline{t}-4\rho^2,\overline{t}+4\rho^2]\subseteq\Omega\times(0,T)$. Then there exists a constant $C=C(n,p)>0$, such that
$$
\frac{|\nabla u|}{u}\leq C\left(\frac{1}{\rho}+\sqrt{p D^{p-1}}\right)\left(1+\log\frac{D}{u}\right)
$$ 
in $Q=B(\overline{x},\rho)\times(\overline{t}-\rho^2,\overline{t}+\rho^2)$. In particular, for every $\delta>0$,  there exists a constant $C=C(n,p,\delta)>0$, such that
$$
|\nabla u|\leq C\left(\frac{1}{\rho}+\sqrt{p D^{p-1}}\right) u^{1-\delta}
$$
in $Q$.
\end{lemma}

To keep track of the dependencies of the constants, we define
\begin{equation}\label{gestK}
K= C\left(\frac{1}{\rho}+\sqrt{p D^{p-1}}\right)
\end{equation}
where $C=C(n,p,\delta)>0$ is given in the previous proposition. In particular, we have the gradient estimate
\begin{equation}\label{gest}
|\nabla u|\leq K u^{1-\delta}.
\end{equation}
Notice that $K$ actually depends only on $n,p,\delta,D$ and $\rho_0$, if $\rho$ is larger than some fixed constant $\rho_0>0$.

\begin{remark}\label{r-g}
The same result also holds on complete Riemannian manifolds with nonnegative Ricci tensor (see~\cite{cama2}).
\end{remark}

We recall the following technical lemma~\cite[Lemma~3.1]{bi-ve-ra}.

\begin{lemma}\label{l-lemma2}
Let $U$ be an open set of $\RR^{n}$. Then, for any positive function $w\in C^{2}(U)$, any nonnegative $\eta \in C^{2}_{c}(U)$, any real numbers $d,m\in\RR$ such that $d\neq m+2$, the following inequality holds:
\begin{align*}
&\frac{2(n-m)d-(n-1)(m^2+d^2)}{4n}\int_{U}\eta w^{m-2}|\nabla w|^{4}-\frac{n-1}{n}\int_{U}\eta w^m(\Delta w)^{2}\\
&\quad\qquad-\frac{2(n-1)m+(n+2)d}{2n}\int_{U}\eta w^{m-1}|\nabla w|^{2}\Delta w \\
&\leq \frac{m+d}{2}\int_{U}w^{m-1}|\nabla w|^{2}\langle\nabla w, \nabla \eta\rangle + \int_{U}w^m\Delta w \langle\nabla w, \nabla \eta\rangle+\frac12 \int_{U}w^m|\nabla w|^{2}\Delta\eta.
\end{align*}
\end{lemma}

\begin{remark}\label{r-bv}
The proof is based on the Bochner formula
$$
\frac12 \Delta|\nabla f|^2=|\nabla^2 f|^2+\langle \nabla f,  \nabla \Delta f\rangle \geq \frac{1}{n}(\Delta f)^2+\langle \nabla f, \nabla\Delta f\rangle,
$$
hence, the conclusion of Lemma~\ref{l-lemma2} is also true on Riemannian manifolds with nonnegative Ricci tensor, since this inequality holds in such spaces (see~\cite{petersen1}, for instance).
\end{remark}

From the previous lemma and the pointwise gradient inequality~\eqref{gest}, we can show the following integral estimate which generalizes the one of Lemma~3 in~\cite{bi-ve} (we will use the same notation and line of proof).

\begin{proposition}\label{p-lemma3}
Let $\Omega\subseteq\RR^n$, with $n\geq 3$ and $u\in C^{2,1}(\Omega\times(0,T))$ be a solution of equation~\eqref{eq-pde}, with $1<p<p_S$ and such that $0<u\leq D$. We assume that for some $\rho>0$ we have $\overline{B}(\overline{x},2\rho)\times[\overline{t}-4\rho^2,\overline{t}+4\rho^2]\subseteq\Omega\times(0,T)$ and we let $Q=B(\overline{x},\rho)\times(\overline{t}-\rho^2,\overline{t}+\rho^2)$, with $\rho$ with larger than some $\rho_0>0$. Suppose that $\zeta\in C^{2}_{c}(Q)$ takes values in $[0,1]$ and let $r>4$, $\delta\in(0,1/100)$, then there exist constants $m=m(n,p)>-\frac{2}{n-2}$ and $C=C(n,p,r,\delta,D,\rho_0)>0$ such that
\begin{align}
\int_{Q}\zeta^{r}u^{m}u_{t}^{2}&\,+\int_{Q}\zeta^r u^{m-2}|\nabla u|^{4} + \int_Q \zeta^r u^{m+p-1}|\nabla u|^2 \\
&\leq C\int_{Q}\zeta^{r-4}u^{m+2(1-\delta)}\left(|\nabla \zeta|^{2}+|\zeta_{t}|+(\Delta\zeta)^{2}+|\nabla \zeta|^{4}+\zeta_{t}^{2}\right)\nonumber\\
&\quad+C\int_{Q}\zeta^{r}u^{m+1-2\delta} |u^-_{t}|\label{mainest}
\end{align}
where $u_t^-=u_t\vee 0$. In particular, the constants $m$ and $C$ are independent of $\rho>\rho_0$.
\end{proposition}
\begin{proof}
By sake of clarity, we set $U=B(\overline{x},\rho)$, $t_1=\overline{t}-\rho^2$ and $t_2=\overline{t}+\rho^2$.\\
Applying Lemma~\ref{l-lemma2} to $w=u(\cdot,t)$ and $\eta=\zeta^r(\cdot,t)$, for any $t\in[t_1,t_2]$ and any reals $d,m$ with $d\neq m+2$, we get
\begin{align*}
&\frac{2(n-m)d-(n-1)(m^2+d^2)}{4n}\int_{U}\zeta^r u^{m-2}|\nabla u|^{4}-\frac{n-1}{n}\int_{U}\zeta^r u^m(\Delta u)^{2}\\
&\quad\qquad-\frac{2(n-1)m+(n+2)d}{2n}\int_{U}\zeta^r u^{m-1}|\nabla u|^{2}\Delta u \\
&\leq \frac{m+d}{2}\int_{U}u^{m-1}|\nabla u|^{2}\langle\nabla u, \nabla (\zeta^r)\rangle + \int_{U}u^m\Delta u \langle\nabla u, \nabla (\zeta^r)\rangle+\frac12 \int_{U}u^m|\nabla u|^{2}\Delta(\zeta^r).
\end{align*}
Hence, substituting $\Delta u=u_t-u^p$, we obtain
\begin{align}
&\frac{2(n-m)d-(n-1)(m^2+d^2)}{4n}\int_{U}\zeta^r u^{m-2}|\nabla u|^{4}\nonumber\\
&\quad\qquad-\frac{n-1}{n}\int_{U}\zeta^r u^mu_t\Delta u+\frac{n-1}{n}\int_{U}\zeta^r u^{m+p}\Delta u\nonumber\\
&\quad\qquad-\frac{2(n-1)m+(n+2)d}{2n}\int_{U}\zeta^r u^{m-1}|\nabla u|^{2}u_t\nonumber\\
&\quad\qquad+\frac{2(n-1)m+(n+2)d}{2n}\int_{U}\zeta^r u^{m+p-1}|\nabla u|^{2}\nonumber\\
&\leq \frac{m+d}{2}\int_{U}u^{m-1}|\nabla u|^{2}\langle\nabla u, \nabla (\zeta^r)\rangle+\int_{U}u^mu_t\langle\nabla u, \nabla (\zeta^r)\rangle\nonumber\\
&\quad\qquad-\int_{U}u^{m+p}\langle\nabla u, \nabla (\zeta^r)\rangle+\frac12 \int_{U}u^m|\nabla u|^{2}\Delta(\zeta^r).\label{ineq333}
\end{align}
Integrating by parts in the two integrals in the second line above, we get
\begin{align*}
-\frac{n-1}{n}&\int_{U}\zeta^r u^mu_t\Delta u+\frac{n-1}{n}\int_{U}\zeta^r u^{m+p}\Delta u\\
=&\,\frac{n-1}{n}\int_U \zeta^r u^m \langle \nabla u, \nabla u_t \rangle
+\frac{m(n-1)}{n}\int_U \zeta^r u^{m-1}|\nabla u|^2 u_t\\
&\,+\frac{n-1}{n}\int_U u^m u_t \langle \nabla u,\nabla(\zeta^r)\rangle-\frac{n-1}{n}\int_U  u^{m+p}\langle \nabla u, \nabla(\zeta^r)\rangle\\
&\,-\frac{(m+p)(n-1)}{n}\int_U \zeta^r u^{m+p-1} |\nabla u|^2.
\end{align*}
Thus, substituting and setting
$$
X=\int_U \zeta^r u^m \langle \nabla u, \nabla u_t \rangle,\quad Y=\int_U \zeta^r u^{m-1}|\nabla u|^2 u_t,
$$
$$
Z=\int_U u^m u_t \langle \nabla u,\nabla(\zeta^r)\rangle, \quad V=\int_U  u^{m+p}\langle \nabla u, \nabla(\zeta^r)\rangle,
$$
$$
W=\int_{U}u^{m-1}|\nabla u|^{2}\langle\nabla u, \nabla (\zeta^r)\rangle,\quad R=\int_{U}u^m|\nabla u|^{2}\Delta(\zeta^r),
$$
the inequality~\eqref{ineq333} becomes
\begin{align*}
&\frac{2(n-m)d-(n-1)(m^2+d^2)}{4n}\int_{U}\zeta^r u^{m-2}|\nabla u|^{4}\\
&\quad\qquad+\frac{n-1}{n}X+\frac{m(n-1)}{n}Y+\frac{n-1}{n}Z-\frac{n-1}{n}V\\
&\quad\qquad-\frac{(m+p)(n-1)}{n}\int_U \zeta^r u^{m+p-1} |\nabla u|^2-\frac{2(n-1)m+(n+2)d}{2n}Y\\
&\quad\qquad+\frac{2(n-1)m+(n+2)d}{2n}\int_{U}\zeta^r u^{m+p-1}|\nabla u|^{2}\\
&\leq \frac{m+d}{2}W+Z-V+\frac{1}{2}R,
\end{align*}
Hence, rearranging and simplifying, we conclude
\begin{align}
\alpha\int_{U}\zeta^r u^{m-2}&|\nabla u|^{4} + \beta \int_U \zeta^r u^{m+p-1}|\nabla u|^2\nonumber\\
&\leq -\frac{n-1}{n}X+\frac{(n+2)d}{2n}Y+\frac{1}{n}Z-\frac{1}{n}V+\frac{m+d}{2}W+\frac12 R,\label{eq-ab}
\end{align}
where
$$
\alpha=\frac{2(n-m)d-(n-1)(m^2+d^2)}{4n}\qquad\text{and}\qquad \beta=\frac{(n+2)d-2(n-1)p}{2n}.
$$
We now choose 
$$
m=-\frac{d}{n-1}
$$
obtaining
$$
\alpha=\frac{\left[2(n-1)-(n-2)d\right]d}{4(n-1)}\qquad\text{and}\qquad \beta=\frac{(n+2)d-2(n-1)p}{2n}.
$$
It is then easy to see that both constants $\alpha$ and $\beta$ are positive if $d$ satisfies
\begin{equation}\label{deq}
\frac{2(n-1)p}{n+2}<d<\frac{2(n-1)}{n-2}
\end{equation}
which is a meaningful condition, since $p<p_S=\frac{n+2}{n-2}$. Thus, we set $d$ to be equal to the average mean of the two values, that is
$$
d=\frac{(n-1)p}{n+2}+\frac{n-1}{n-2}
$$
and consequently
$$
m=m(n,p)=-\frac{p}{n+2}-\frac{1}{n-2},
$$
which implies $d\not=m+2$. Indeed, $d=m+2$ if and only if $p=\frac{(n+2)(n-4)}{n(n-2)}$ but this value is always smaller than 1, for $n\in\NN$. Moreover, from the right inequality~\eqref{deq} we have
$$
m=-\frac{d}{n-1}>-\frac{2}{n-2}.
$$

To get the thesis we now bound some of the right--hand side terms in inequality~\eqref{eq-ab}.\\
By using Young's inequality, for any $\eps>0$, there exists $C(\eps)>0$ such that
\begin{align}\label{eq-V}
V &\leq \eps \int_{U}\zeta^{r}u^{m+p-1}|\nabla u|^{2} + C(\eps)\int_{U}\zeta^{r-2}u^{m+p+1}|\nabla \zeta|^{2},\\
W &\leq \eps \int_{U} \zeta^{r}u^{m-2}|\nabla u|^{4}+C(\eps)\int_{U}\zeta^{r-4}u^{m+2}|\nabla \zeta|^{4},\label{eq-W}\\
R &= r(r-1)\int_{U}\zeta^{r-2}u^{m}|\nabla u|^{2}|\nabla \zeta|^{2}+r\int_{U}\zeta^{r-1}u^{m}|\nabla u|^{2}\Delta \zeta\nonumber \\
&\leq \eps \int_{U}\zeta^{r}u^{m-2}|\nabla u|^{4}+C(\eps)\int_{U}\zeta^{r-4}u^{m+2}|\nabla \zeta|^{4}+C(\eps)\int_{U}\zeta^{r-2}u^{m+2}(\Delta \zeta)^{2}\nonumber\\
&\leq \eps \int_{U}\zeta^{r}u^{m-2}|\nabla u|^{4}+C(\eps)\int_{U}\zeta^{r-4}u^{m+2}|\nabla \zeta|^{4}+C(\eps)\int_{U}\zeta^{r-4}u^{m+2}(\Delta \zeta)^{2}\nonumber\\
&= \eps \int_{U}\zeta^{r}u^{m-2}|\nabla u|^{4}+C(\eps)\int_{U}\zeta^{r-4}u^{m+2}\bigl(|\nabla \zeta|^{4}+(\Delta \zeta)^{2}\bigr)\label{eq-R}
\end{align}
since $\zeta\leq 1$ everywhere.
Moreover,
\begin{equation}\label{eq-X}
X=\int_U \zeta^r u^m \langle \nabla u, \nabla u_t \rangle=\frac{df}{dt}-P-\frac{m}{2}Y
\end{equation}
with 
$$
f=\frac12\int_{U}\zeta^{r}u^{m}|\nabla u|^{2}\qquad\text{and}\qquad P=\frac12 \int_{U}u^{m}|\nabla u|^{2}(\zeta^{r})_{t}.
$$
Then, by equation $u_t=\Delta u+u^p$ and integrating by parts the Laplacian from the second to the third line as we did before, we get
\begin{align}
\int_{U}\zeta^{r}u^{m}u_{t}^{2}&=\int_{U}\zeta^{r}u^{m}u_{t}\left(\Delta u + u^{p}\right) \nonumber\\
&= \int_{U}\zeta^{r}u^{m+p}u_{t}+\int_{U}\zeta^{r}u^{m}u_t\Delta u\nonumber\\
&= \frac{dg}{dt}-S-Z-X-mY\nonumber\\
&= \frac{d(g-f)}{dt}+P-Z-S-\frac{m}{2}Y,\label{eq444}
\end{align}
where
$$
g=\frac{1}{m+p+1}\int_{U}\zeta^{r}u^{m+p+1}\qquad\text{and}\qquad S=\frac{1}{m+p+1}\int_{U}u^{m+p+1}(\zeta^{r})_{t},
$$
which implies
\begin{equation}\label{eq-Z}
Z=\frac{d(g-f)}{dt}+P-S-\frac{m}{2}Y-\int_{U}\zeta^{r}u^{m}u_{t}^{2}.
\end{equation}
We bound $P$ as follows,
\begin{equation}\label{eq-P}
P =\frac{r}{2}\int_{U}\zeta^{r-1}\zeta_{t}u^{m}|\nabla u|^{2}\leq \eps\int_{U}\zeta^r u^{m-2}|\nabla u|^{4}+C(\eps)\int_{U}\zeta^{r-2}u^{m+2}\zeta_{t}^{2}.
\end{equation}
Substituting equalities~\eqref{eq-X},~\eqref{eq-Z} and inequalities~\eqref{eq-V},~\eqref{eq-W},~\eqref{eq-R} into estimate~\eqref{eq-ab}, we obtain
\begin{align*}
\alpha\int_{U}\zeta^r u^{m-2}&|\nabla u|^{4} +\beta \int_U \zeta^r u^{m+p-1}|\nabla u|^2 \\
\leq&\, -\frac{n-1}{n}\Bigl(\frac{df}{dt}-P-\frac{m}{2}Y\Bigr)+\frac{(n+2)d}{2n}Y\\
&\,+\frac{1}{n} \Bigl(\frac{d(g-f)}{dt}+P-S-\frac{m}{2}Y-\int_{U}\zeta^{r}u^{m}u_{t}^{2}\Bigr)\\
&\,+ \frac{\eps}{n} \int_{U}\zeta^{r}u^{m+p-1}|\nabla u|^{2}+\frac{(m+d+1)\eps}{2}\int_{U} \zeta^{r}u^{m-2}|\nabla u|^{4}\\
&\,+C\int_{U}\zeta^{r-4}u^{m+2}\left((\Delta\zeta)^{2}+|\nabla \zeta|^{4}\right)+C\int_{U}\zeta^{r-2}u^{m+p+1}|\nabla \zeta|^{2}\\
\leq&\, -\frac{df}{dt} +\frac{1}{n}\frac{dg}{dt}+\Bigl(\frac{(n-1)m}{2n}+\frac{(n+2)d}{2n}-\frac{m}{2n}\Bigr)Y-\frac{1}{n}\int_{U}\zeta^{r}u^{m}u_{t}^{2}\\
&\,+ \frac{\eps}{n} \int_{U}\zeta^{r}u^{m+p-1}|\nabla u|^{2}+\frac{(m+d+3)\eps}{2}\int_{U} \zeta^{r}u^{m-2}|\nabla u|^{4}\\
&\,+C\int_{U}\zeta^{r-4}u^{m+2}\left((\Delta\zeta)^{2}+|\nabla \zeta|^{4}+\zeta_{t}^{2}\right)+C\int_{U}\zeta^{r-2}u^{m+p+1}\left(|\nabla \zeta|^{2}+|\zeta_{t}|\right)
\end{align*}
for some constant $C=C(n,p,r,\varepsilon)$, where we estimated $S$ simply taking the modulus of the integrand and used inequality~\eqref{eq-P} to deal with $P$.
We notice that the coefficient of the term $Y$ is given by
$$
\frac{(n-1)m}{2n}+\frac{(n+2)d}{2n}-\frac{m}{2n}=\frac{nd}{2(n-1)}>0
$$
as $m=-\frac{d}{n-1}$.\\
Taking $\eps$ small enough and then ``absorbing'' the two integrals in $\nabla u$ in the left side of the inequality, we can conclude that there exists a constant $C_1$ depending only on $n,p,r$, such that for every $t\in (t_1,t_2)$ we have
\begin{align*}
\frac{\alpha}{2}\int_{U}\zeta^r u^{m-2}&|\nabla u|^{4} + \frac{\beta}{2}\int_U \zeta^r u^{m+p-1}|\nabla u|^2 \\
\leq&\, -\frac{df}{dt} +\frac{1}{n}\frac{dg}{dt}+\frac{nd}{2(n-1)}Y-\frac{1}{n}\int_{U}\zeta^{r}u^{m}u_{t}^{2}.\\
&\,+C_1\int_{U}\zeta^{r-4}u^{m+2}\left((\Delta\zeta)^{2}+|\nabla \zeta|^{4}+\zeta_{t}^{2}\right)+C_1\int_{U}\zeta^{r-2}u^{m+p+1}\left(|\nabla \zeta|^{2}+|\zeta_{t}|\right)
\end{align*}
(notice that possibly varying the constant $C_1$, instead of the constants $\alpha/2$ and $\beta/2$ in front of the first two integrals we could have chosen $\alpha-\delta$ and $\beta-\delta$, for any $\delta>0$).\\
Integrating this inequality between $t_{1}$ and $t_{2}$ and observing that $f(t_{i})=g(t_{i})=0$, for $i=1,2$, since $\zeta\in C^{2}_{c}(Q)$, we get
\begin{align}
\frac{\alpha}{2}\int_{Q}\zeta^r u^{m-2}&|\nabla u|^{4} +\frac{\beta}{2}\int_{Q} \zeta^r u^{m+p-1}|\nabla u|^2\nonumber \\
\leq&\,C_1\int_{Q}\zeta^{r-4}u^{m+2}\left((\Delta\zeta)^{2}+|\nabla \zeta|^{4}+\zeta_{t}^{2}\right)+C_1\int_{Q}\zeta^{r-2}u^{m+p+1}\left(|\nabla \zeta|^{2}+|\zeta_{t}|\right)\nonumber\\
&+\frac{nd}{2(n-1)}\int_{t_1}^{t_2}Y-\frac{1}{n}\int_{Q}\zeta^{r}u^{m}u_{t}^{2}.\label{eq333}
\end{align}
Finally, in order to estimate $\int_{t_1}^{t_2}Y$, we make use of the gradient estimate~\eqref{gest} from Lemma~\ref{p-grad}. For any $\delta>0$ we have
\begin{align*}
Y =&\,\int_U \zeta^r u^{m-1}|\nabla u|^2 u_t\\
\leq&\,\int_{U} \zeta^r u^{m-1}|\nabla u|^2 u^+_t \\
\leq&\, K \int_{U} \zeta^{r}u^{m+1-2\delta} u^+_{t},
\end{align*}
where $K$ is defined in formula~\eqref{gestK} and we set $u_t=u^+_t+u^-_t$, $u_t^+=u_t\wedge 0\geq 0$ and $u^-_t=u_t\vee 0\leq0$. It follows
\begin{align}\label{eq-iY}\nonumber
\int_{t_1}^{t_2} Y
&\leq K \int_{U}\int_{t_1}^{t_2} \zeta^{r}u^{m+1-2\delta} u^+_{t}\\
&=K\int_{U}\int_{t_1}^{t_2} \zeta^{r}u^{m+1-2\delta} u_{t} -K\int_{U}\int_{t_1}^{t_2} \zeta^{r}u^{m+1-2\delta} u^-_{t}\nonumber\\
&= K\int_{U}\int_{t_1}^{t_2} \left(\zeta^{r}u^{m+2(1-\delta)}\right)_t - K \int_{U}\int_{t_1}^{t_2}u^{m+2(1-\delta)}(\zeta^{r})_{t}-K\int_{U}\int_{t_1}^{t_2} \zeta^{r}u^{m+1-2\delta} u^-_{t}\nonumber\\
&\leq K \int_Q \zeta^{r-1}u^{m+2(1-\delta)}|\zeta_{t}|-K\int_{Q} \zeta^{r}u^{m+1-2\delta} u^-_{t}
\end{align}
since $\zeta(x,t_i)=0$, for $i=1,2$ and every $x\in U$.\\
Substituting in inequality~\eqref{eq333} we get that there exists a constant $C_2=C_2(n,p,r,\delta,D,\rho_0)$ (since we have seen after Lemma~\ref{p-grad} that $K$ depends on $n,p,\delta,D$ and $\rho_0$) such that
\begin{align*}
\frac{1}{n}\int_{Q}\zeta^{r}u^{m}u_{t}^{2}&\,+\frac{\alpha}{2}\int_{Q}\zeta^r u^{m-2}|\nabla u|^{4} +\frac{\beta}{2}\int_Q \zeta^r u^{m+p-1}|\nabla u|^2 \\
&\leq C_2\Bigl(\int_{Q}\zeta^{r-2}u^{m+p+1}\left(|\nabla \zeta|^{2}+|\zeta_{t}|\right)+\int_{Q}\zeta^{r-4}u^{m+2}\left((\Delta\zeta)^{2}+|\nabla \zeta|^{4}+\zeta_{t}^{2}\right)\Bigr)\\
&\quad+C_2\int_{Q}\zeta^{r-1}u^{m+2(1-\delta)}|\zeta_{t}| -C_2\int_{Q} \zeta^{r}u^{m+1-2\delta} u^-_{t}.
\end{align*}
Since $u\leq D$, $0\leq\zeta\leq1$ and $m+p+1>m+2>m+2(1-\delta)>0$, for $\delta$ sufficiently small, as $m>-\frac{2}{n-2}$ and $n\geq 3$, 
we have $u^{m+p+1}\leq C u^{m+2(1-\delta)}$, then there exists a positive constant $C_3=C_3(n,p,r,\delta, D,\rho_0)$ such that 
\begin{align*}
\int_{Q}\zeta^{r}u^{m}u_{t}^{2}&\,+\int_{Q}\zeta^r u^{m-2}|\nabla u|^{4} + \int_Q \zeta^r u^{m+p-1}|\nabla u|^2 \\
&\leq C_3\int_{Q}\zeta^{r-4}u^{m+2(1-\delta)}\left(|\nabla \zeta|^{2}+|\zeta_{t}|+(\Delta\zeta)^{2}+|\nabla \zeta|^{4}+\zeta_{t}^{2}\right)\\
&\quad -C_3\int_{Q}\zeta^{r}u^{m+1-2\delta} u^-_{t},
\end{align*}
which is the thesis.
\end{proof}

\begin{remark}\label{p-lemma3-n12}
The same proposition also holds for $n=1,2$, for every $p>1$, considering $m=0$, $d=1$, if $n=1$ and $m=-d$ with any $d>2(n-1)p/(n+2)$, when $n=2$.
\end{remark}

Now we see that the estimate of this proposition implies an interior integral estimate on $u^{2p+m}$ (see~\cite[Lemma~4]{bi-ve}).

\begin{lemma}\label{l-lemma4}
In the same setting and with $m$ as in Proposition~\ref{p-lemma3}, if $r>\frac{2(2p+m)}{p-1+\delta}$ there exists a constant $C=C(n,p,r,\delta,D,\rho_0)>0$ such that
\begin{align}
\int_{Q}\zeta^{r}u^{2p+m} \leq&\, C\int_{Q}\zeta^{r-\frac{2(2+m)}{p-1+\delta}}\left(|\nabla \zeta|^{2}+|\zeta_{t}|+(\Delta\zeta)^{2}+|\nabla \zeta|^{4}+\zeta_{t}^{2}\right)^{\frac{2p+m}{2p-2(1-\delta)}}\nonumber\\
&\,+C\int_{Q}\zeta^{r}u^{m+1-2\delta} |u^-_{t}|.\label{eq445}
\end{align}
\end{lemma}
\begin{proof}
Multiplying equation~\eqref{eq-pde} by $\zeta^{r}u^{m+p}$ and integrating on $U$, we obtain
\begin{align*}
\int_{U}\zeta^{r}u^{2p+m} &= \int_{U}\zeta^{r}u^{m+p}u_{t}-\int_{U}\zeta^{r}u^{m+p}\Delta u \\
&= \frac{dg}{dt}-S+V+(m+p)\int_{U}\zeta^{r}u^{m+p-1}|\nabla u|^{2}.
\end{align*}
Estimating $S$ and $V$ as in the proof of Proposition~\ref{p-lemma3}, then integrating between $t_{1}$ and $t_{2}$ and using estimate~\eqref{mainest}, we get
\begin{align*}
\int_{Q}\zeta^{r}u^{2p+m} &\leq C\int_{Q}\zeta^{r-4}u^{m+2(1-\delta)}\left(|\nabla \zeta|^{2}+|\zeta_{t}|+(\Delta\zeta)^{2}+|\nabla \zeta|^{4}+\zeta_{t}^{2}\right)\\
&\quad +C\int_{Q}\zeta^{r}u^{m+1-2\delta} |u^-_{t}|
\end{align*}
with a constant $C=C(n,p,r,\delta,D,\rho_0)$.
Observing that
$$
2p+m>m+2(1-\delta),
$$
the conclusion of the lemma follows easily by means of Young's inequality.
\end{proof}

We then have the following integral decay estimate if $u$ is ``pointwise'' monotone nondecreasing in time, that is $u_t\geq0$ everywhere.

\begin{proposition}\label{p-intdec}
Let $\Omega\subseteq\RR^n$ and $u\in C^{2,1}(\Omega\times(0,T))$ be a positive solution of equation~\eqref{eq-pde} with $p<p_S=\frac{n+2}{n-2}$, such that $0<u\leq D$ and $u_t\geq0$. Let $0<\overline{t}<T$ and $\overline{x}\in \Omega$, for $m>-\frac{2}{n-2}$ as in Proposition~\ref{p-lemma3}, $\delta\in(0,1/100)$, $r>\frac{2(2p+m)}{p-1+\delta}$ and $\rho_0>0$, there exists a constant $C=C(n,p,r,\delta,D,\rho_0)>0$ such that, for every $\rho>\rho_0$ with $\overline{B}(\overline{x},2\rho)\times[\overline{t}-4\rho^2,\overline{t}+4\rho^2]\subseteq\Omega\times(0,T)$, letting $Q(\rho)=B(\overline{x},\rho)\times(\overline{t}-\rho^{2},\overline{t}+\rho^{2})$, there holds
$$
\int_{Q(\rho/2)} u^{2p+m} \leq C \rho^{n+2-\frac{2p+m}{p-1+\delta}}.
$$
\end{proposition}
\begin{proof}
Since $u_t^-=0$, we obtain the thesis by applying Lemma~\ref{l-lemma4}, choosing $\zeta(x,t)=\varphi(x)\psi(t)$ with $\varphi\in C^{2}_{c}(B(\overline{x},\rho))$ and $\psi\in C^{1}_{c}(\overline{t}-\rho^{2},\overline{t}+\rho^{2})$, both taking values in $[0,1]$, such that $\varphi=1$ on $B(\overline{x},\rho/2)$, $\psi= 1$ on $[\overline{t}-\rho^{2}/4,\overline{t}+\rho^{2}/4]$ and 
\begin{equation}\label{eq-cutoff}
|\Delta \varphi|+|\nabla \varphi|^{2}+|\psi_{t}| \leq C(n) \rho^{-2}. 
\end{equation}
\end{proof}

\begin{remark}\label{r-cutoff}
Given a complete Riemannian manifold with nonnegative Ricci tensor, it is possible to construct cutoff functions satisfying conditions~\eqref{eq-cutoff} (see~\cite{guneysu}, for instance). In particular, using Remarks~\ref{r-g} and~\ref{r-bv}, it is then clear that the same proof goes through also in this case. 
\end{remark}

\begin{proposition}\label{p-intdec-cor}
Let $u\in C^{2,1}(\Omega\times(0,T))$ be a positive bounded solution of equation~\eqref{eq-pde} with $p<p_S=\frac{n+2}{n-2}$, such that $0<u\leq D$ and $u_t\geq0$, where $\Omega$ is an open subset of a Riemannian manifold $(M,g)$ with nonnegative Ricci tensor. Let $0<\overline{t}<T$ and $\overline{x}\in \Omega$. For $m>-\frac{2}{n-2}$ as in Proposition~\ref{p-lemma3}, $\delta\in(0,1/100)$, $r>\frac{2(2p+m)}{p-1+\delta}$ and $\rho_0>0$, there exists a constant $C=C(n,p,r,\delta,D,\rho_0)>0$ such that, for every $\rho>\rho_0$ with $\overline{B}(\overline{x},2\rho)\times[\overline{t}-4\rho^2,\overline{t}+4\rho^2]\subseteq\Omega\times(0,T)$, letting 
$Q(\rho)=B(\overline{x},\rho)\times(\overline{t}-\rho^{2},\overline{t}+\rho^{2})$, there holds
$$
\int_{Q(\rho/2)} u^{2p+m} \leq C\rho^{2-\frac{2p+m}{p-1+\delta}}\mu(B(\overline{x},\rho)),
$$
where $B(\overline{x},\rho)\subseteq M$ is the metric ball with center $\overline{x}$ and radius $\rho$ and $\mu$ the canonical Riemannian measure of $(M,g)$.
\end{proposition}

\begin{remark} 
The assumption that $u$ is ``pointwise'' monotone nondecreasing in time (that is, $u_t^-$ is identically zero, hence the second integral in the right hand side of formula~\eqref{eq445} in Lemma~\ref{l-lemma4} vanishes) can be weakened a little. 
Indeed, notice that if $\delta>0$ is sufficiently small, we have
$$
m+1-2\delta>-\frac{2}{n-2}+1-2\delta>0,
$$
when $n>4$, hence, using Young's inequality, we get
$$
C\int_{Q}\zeta^ru^{m+1-2\delta} |u^-_{t}|\leq \frac12\int_{Q}\zeta^r u^{2p+m}+\overline{C}\int_Q\zeta^r|u^-_t|^{\frac{2p+m}{2p-1+2\delta}}
$$
where $C$ is the constant in formula~\eqref{eq445}. Then, we obtain
\begin{align*}
\frac{1}{2}\int_{Q}\zeta^{r}u^{2p+m} \leq&\, C\int_{Q}\zeta^{r-\frac{2(2+m)}{p-1+\delta}}\left(|\nabla \zeta|^{2}+|\zeta_{t}|+(\Delta\zeta)^{2}+|\nabla \zeta|^{4}+\zeta_{t}^{2}\right)^{\frac{2p+m}{2p-2(1-\delta)}}\nonumber\\
&\,+\overline{C}\int_Q|u^-_t|^{\frac{2p+m}{2p-1+2\delta}},
\end{align*}
being $0\leq\zeta\leq 1$.\\
Hence, if there exists a constant $\widetilde{C}=\widetilde{C}(n,p,r,\delta,D,\rho_0)$, such that, for $\rho>\rho_0$ with $\overline{B}(\overline{x},2\rho)\times[\overline{t}-4\rho^2,\overline{t}+4\rho^2]\subseteq\Omega\times(0,T)$, there holds
\begin{equation}\label{extraAss}
\int_Q|u^-_t|^{\frac{2p+m}{2p-1+2\delta}} \leq \widetilde{C}\rho^{2-\frac{2p+m}{p-1+\delta}}\mu(B(\overline{x},\rho)),
\end{equation}
then we have the same conclusion of the above proposition, along the same line of proof.
\end{remark}

\section{Weak Harnack inequality and the proof of Theorem~\ref{t-triv}}

We are going to apply the following estimates of Aronson--Serrin in~\cite{aroser}, which have been extended to manifolds with Ricci tensor bounded from below by Saloff--Coste in~\cite{saloff}. For $\Omega\subseteq\RR^n$, let $u\in C^{2,1}(\Omega\times(0,T))$ be a positive bounded solution of equation~\eqref{eq-pde}, with $p<p_S=\frac{n+2}{n-2}$. Following their notation, we set
$$
d=u^{p-1}\qquad\qquad\text{hence,}\qquad\qquad u_t=\Delta u+u^p = \Delta u + d u.
$$
Fixed a pair $(\overline{x}, \overline{t})\in \Omega\times(0,T)$ and $\rho>0$, we consider the parabolic cylinder
$$
\widetilde{Q}(\rho)=B(\overline{x},\rho)\times (\overline{t}-\rho^2,\overline{t})\subseteq Q(\rho)
$$
(where $Q(\rho)=B(\overline{x},\rho)\times (\overline{t}-\rho^2,\overline{t}+\rho^2)$ was the set defined in the previous section) and we will use the symbol $\Vert\cdot\Vert_{q,\rho}$ to denote the $L^q$--norm of a function in $\widetilde{Q}(\rho)$. 
In~\cite{aroser} the following weak Harnack estimate is proved (we mention that the assumption $q>\frac{n+2}{2}$ guarantees the validity of the assumption $(4)$ in~\cite{aroser}, necessary for~Lemma~3).

\begin{theorem}[{\cite[Theorem~2]{aroser}}]\label{t-aroser}
For $\Omega\subseteq\RR^n$, let $u\in C^{2,1}(\Omega\times(0,T))$ be a solution of equation $u_t=\Delta u + d u$. Suppose that $\widetilde{Q}(3\rho)\subseteq \Omega\times(0,T)$, then, for any $q>\frac{n+2}{2}$, there exists a constant $C=C(n,q,\rho,d)>0$ such that
$$
|u(x,t)| \leq C \rho^{-\frac{n+2}{2}}\,\Vert u \Vert_{2,3\rho}
$$
for every $(x,t)\in \widetilde{Q}(\rho)$.
\end{theorem} 

Inspecting carefully the proof and keeping track of the dependencies of the constants, it follows that the constant $C$ can be made more explicit, that is, 
$$
C=C(n)\rho^{1-\frac{n+2}{2q}}\Vert d \Vert_{q,3\rho}^{1/2},
$$
hence, in the same hypotheses of this theorem, we can conclude
$$
|u(x,t)| \leq C \rho^{1-\frac{n+2}{2q}-\frac{n+2}{2}}\Vert d \Vert_{q,3\rho}^{1/2}\, \Vert u \Vert_{2,3\rho}\leq C \rho \, |\widetilde{Q}(3\rho)|^{-\frac{q+1}{2q}} \Vert d \Vert_{q,3\rho}^{1/2}\, \Vert u \Vert_{2,3\rho},
$$
for every $(x,t)\in \widetilde{Q}(\rho)$, with $C=C(n)$, as $|\widetilde{Q}(3\rho)| =(3 \rho)^2 |B(\overline{x},3\rho)|=\omega_n 3^{n+2} \rho^{n+2}$ is the Lebesgue measure of the parabolic cylinder $\widetilde{Q}(3\rho)$.

The same estimate can be extended to Riemannian manifolds $(M,g)$ with nonnegative Ricci tensor, see~\cite[Theorem~5.6]{saloff} (with $K=0$), obtaining
\begin{equation*}
|u(x,t)| \leq C \rho^{-\frac1q} \, \mu(B(\overline{x},3\rho))^{-\frac{q+1}{2q}} \Vert d \Vert_{q,3\rho}^{1/2}\, \Vert u \Vert_{2,3\rho},
\end{equation*}
where $C=C(n,M)$ and $B(\overline{x},\rho)\subseteq M$ is the metric ball with center $\overline{x}$ and radius $\rho$ and $\mu$ the canonical Riemannian measure of $(M,g)$.

Setting then $d=u^{p-1}$ in this estimate we get the following corollary.

\begin{corollary}\label{t-aroserc}
Let $(M,g)$ be a complete Riemannian manifold with nonnegative Ricci tensor. Let $u\in C^{2,1}(M\times\RR)$ be a nonnegative solution of equation $u_t=\Delta u+u^p$ with $p>1$, then for any $q>\frac{n+2}{2}$ there exists a positive constant $C=C(n,M)>0$ such that in $\widetilde{Q}(\rho)$ we have the estimate,
$$
u(x,t) \leq C \rho^{-\frac1q} \, \mu(B(\overline{x},3\rho))^{-\frac{q+1}{2q}} \Vert u \Vert_{q(p-1),3\rho}^{\frac{p-1}{2}}\, \Vert u \Vert_{2,3\rho}
$$
for every $\rho>0$.
\end{corollary} 

We are then ready to show Theorem~\ref{t-triv}.

\begin{proof}[Proof of Theorem~\ref{t-triv}] Let $(M,g)$ be a complete Riemannian manifold with nonnegative Ricci tensor. Let $u\in C^{2,1}(M\times\RR)$ be a bounded eternal solution of equation~\eqref{eq-pde} such that $u_t \geq 0$, with $1<p<p_S=\frac{n+2}{n-2}$ and $n\geq 3$, then, by Theorem~2.6 in ~\cite{cama2}, the function $u$ is identically zero or positive everywhere. The conclusion in the case $1<p<p_{BV}=\frac{n(n+2)}{(n-1)^2}$ is a consequence of the estimates in~\cite{bi-ve} (see Theorem~A in~\cite{pqs}) for the Euclidean setting, which can be extended to Riemannian manifolds with nonnegative Ricci tensor as observed in Remark~\ref{rem-bv}, in the introduction. Thus, we assume 
$$
\frac{n(n+2)}{(n-1)^2}\leq p<p_S.
$$
We choose $m=-\frac{d}{n-1}>-\frac{2}{n-2}$, as in Proposition~\ref{p-lemma3} and we observe that 
$$
2p+m\geq\frac{2n(n+2)}{(n-1)^2}-\frac{2}{n-2}=2\frac{n^3-n^2-2n-1}{(n-2)(n-1)^2}=2+\frac{2(3n^2-7n+1)}{(n-2)(n-1)^2}>2
$$
when $n\geq 3$.\\
Now let $\gamma>0$ such that 
\begin{equation}\label{eq-c1}
q=(1+\gamma)\frac{2p+m}{p-1}>\frac{n+2}{2}
\end{equation}
and 
\begin{equation}\label{eq-c2}
1+\frac{2(1+\gamma)}{1+2\gamma}>p,
\end{equation}
(we will discuss the existence of such $\gamma$ later on).\\
It follows that  
$$ 
\lim_{\delta\to0^+}\Bigl(1-\frac{p-1}{2(1+\gamma)(p-1+\delta)}-\frac{1}{p-1+\delta}\Bigr)={\frac{1+2\gamma}{2(1+\gamma)}-\frac{1}{p-1}}<0,
$$
hence, the quantity 
\begin{equation}\label{epsform}
\sigma=1-\frac{p-1}{2(1+\gamma)(p-1+\delta)}-\frac{1}{p-1+\delta}
\end{equation}
is negative, choosing $\delta>0$ sufficiently small.\\
Now, fixing $(\overline{x}, \overline{t})\in M\times\RR$ and $\rho>0$, from Holder's inequality and Proposition~\ref{p-intdec-cor} (notice that $\widetilde{Q}(3\rho)\subseteq B(\overline{x},3\rho)\times[\overline{t}-9\rho^2,\overline{t}+9\rho^2]$), we obtain
\begin{align}\nonumber
\Vert u \Vert_{2,3\rho}=&\, \Bigl( \int_{\widetilde{Q}(3\rho)} u^2 \Bigr)^{1/2}\\
\leq&\,\left(\rho^2\mu(B(\overline{x},3\rho))\right)^{\frac{2p+m-2}{2(2p+m)}}\Bigl( \int_{\widetilde{Q}(3\rho)} u^{2p+m} \Bigr)^{\frac{1}{2p+m}}\nonumber\\
\leq&\, C \left(\rho^2\mu(B(\overline{x},3\rho))\right)^{1/2}\rho^{-\frac{1}{p-1+\delta}}\label{eq-L2}
\end{align}
for every $\rho>0$.\\
Then, using the fact that $u$ is bounded and $\gamma>0$, again by Holder's inequality and Proposition~\ref{p-intdec-cor} we have
\begin{align}\nonumber
\Vert u\Vert_{q(p-1),3\rho}^{p-1}
=&\,\Bigl( \int_{\widetilde{Q}(3\rho)} u^{q(p-1)} \Bigr)^{1/q}\nonumber\\
=&\, \Bigl( \int_{\widetilde{Q}(3\rho)} u^{(1+\gamma)(2p+m)} \Bigr)^{1/q}\nonumber\\
\leq&\, C \Bigl( \int_{\widetilde{Q}(3\rho)} u^{2p+m} \Bigr)^{1/q}\nonumber\\
\leq&\, C \left(\rho^2\mu(B(\overline{x},3\rho))\right)^{1/q}\rho^{-\frac{p-1}{(1+\gamma)(p-1+\delta)}}.\label{eq-Lq}
\end{align}
By Corollary~\ref{t-aroserc}, since $q>\frac{n+2}{2}$, there exists a positive constant $C=C(n,M)$ such that in $\widetilde{Q}(\rho)$ we have
\begin{equation*}
u(x,t) \leq C \rho^{-\frac1q} \, \mu(B(\overline{x},3\rho))^{-\frac{q+1}{2q}} \Vert u \Vert_{q(p-1),3\rho}^{\frac{p-1}{2}}\, \Vert u \Vert_{2,3\rho}
\end{equation*}
hence, by estimates~\eqref{eq-L2} and~\eqref{eq-Lq}, we obtain that for every $(x,t)\in \widetilde{Q}(\rho)$ there holds
\begin{align*}
u(x,t) \leq&\, C \rho^{-\frac1q} \, \mu(B(\overline{x},3\rho))^{-\frac{q+1}{2q}} \Vert u \Vert_{q(p-1),3\rho}^{\frac{p-1}{2}}\, \Vert u \Vert_{2,3\rho}\\
\leq&\, C \rho^{-\frac1q} \, \mu(B(\overline{x},3\rho))^{-\frac{q+1}{2q}} \left(\rho^2\mu(B(\overline{x},3\rho))\right)^{\frac{1}{2q}}\rho^{-\frac{p-1}{2(1+\gamma)(p-1+\delta)}} \left(\rho^2\mu(B(\overline{x},3\rho))\right)^{\frac12}\rho^{-\frac{1}{p-1+\delta}} \\
=&\,C \rho^{1-\frac{p-1}{2(1+\gamma)(p-1+\delta)}-\frac{1}{p-1+\delta}}\\
=&\,C\rho^{\sigma}
\end{align*}
where $\sigma<0$ is defined by formula~\eqref{epsform}.\\
Since $u$ is nonnegative, letting $\rho\to+\infty$, we obtain $u(x,t)=0$ for every $(x,t)\in M \times(-\infty,\overline{t})$ and being $\overline{t}\in\RR$ arbitrary, $u=0$ everywhere.

Finally, one can check that, if $n\geq 5$ and $p<p_S$, there exists $\gamma>0$ such that conditions~\eqref{eq-c1} and~\eqref{eq-c2} are satisfied. In fact, in order to satisfy condition~\eqref{eq-c1}, thanks to the fact that $2p+m \geq 2$ and $p < \frac{n+2}{n-2}$, for any $n \geq 3$ we can choose a $\gamma>0$ such that 
\begin{equation*}
\gamma > \frac{n+2}{4} \left( \frac{n+2}{n-2} - 1 \right) -1 = \frac{4}{n-2} .
\end{equation*}
On the other hand, to fulfill condition~\eqref{eq-c2} it is sufficient to choose $\gamma$ that satisfies
\begin{equation*}
\frac{1}{1+2\gamma} > p-2,
\end{equation*}
which, for $p < \frac{n+2}{n-2}$ and $n \geq 6$, is true for any $\gamma >0$. Finally, for $n=5$, it is sufficient to choose $\gamma=1$. 
\end{proof}

\bibliographystyle{amsplain}
\bibliography{biblio}

\end{document}